\documentclass[reqno]{amsproc}
\usepackage{amsfonts}
\usepackage{amssymb}
\usepackage{mathrsfs}
\usepackage{graphicx}
\usepackage{color}

\makeatletter
\@namedef{subjclassname@2010}{%
\textup{2010} Mathematics Subject Classification}
\makeatother

\newtheorem{thm}{Theorem}[section]
\newtheorem{pro}[thm]{Proposition}
\newtheorem{cor}[thm]{Corollary}
\newtheorem{lem}[thm]{Lemma}

\newtheorem*{opq*}{\bf Problem}

\theoremstyle{remark}
\newtheorem{rem}[thm]{Remark}

\theoremstyle{definition}
\newtheorem{dfn}[thm]{Definition}
\newtheorem{exa}[thm]{Example}

\newcommand*{\ascr}{\mathscr A}
\newcommand*{\cbb}{\mathbb C}
\newcommand*{\cfw}{C_{\phi, \mathsf w}}
\newcommand*{\efw}{\mathsf{E}_{\phi,\mathsf w}}
\newcommand*{\Ge}{\geqslant}
\newcommand*{\hfw}{h_{\phi, \mathsf w}}
\newcommand*{\hh}{\mathcal H}
\newcommand*{\Le}{\leqslant}
\newcommand*{\ogr}[1]{\boldsymbol B(#1)}

\newcommand*{\zbb}{\mathbb Z}

\begin{document}
   \title[The Cauchy dual subnormality problem for cyclic
$2$-isometries] {The Cauchy dual subnormality problem
\\ for cyclic $2$-isometries}
   \author[A. Anand]{Akash Anand}
   \address{Department of Mathematics and Statistics\\
Indian Institute of Technology Kanpur, India}
   \email{akasha@iitk.ac.in}
   \author[S. Chavan]{Sameer Chavan}
   \address{Department of Mathematics and Statistics\\
Indian Institute of Technology Kanpur, India}
   \email{chavan@iitk.ac.in}
   \author[Z.\ J.\ Jab{\l}o\'nski]{Zenon Jan
Jab{\l}o\'nski}
   \address{Instytut Matematyki,
Uniwersytet Jagiello\'nski, ul.\ \L ojasiewicza 6,
PL-30348 Kra\-k\'ow, Poland}
\email{Zenon.Jablonski@im.uj.edu.pl}
   \author[J.\ Stochel]{Jan Stochel
   \\\\
{\em \normalsize{Dedicated to Professor Franciszek
Hugon Szafraniec \\ on the occasion of his 80th
birthday}}}
   \address{Instytut Matematyki, Uniwersytet
Jagiello\'nski, ul.\ \L ojasiewicza 6, PL-30348
Kra\-k\'ow, Poland} \email{Jan.Stochel@im.uj.edu.pl}
   \begin{abstract}
The Cauchy dual subnormality problem asks whether the
Cauchy dual operator of a $2$-isometry is subnormal.
Recently this problem has been solved in the negative.
Here we show that it has a negative solution even in
the class of cyclic $2$-isometries.
   \end{abstract}

   \maketitle

   \section{Introduction}
Very recently the Cauchy dual subnormality problem was
solved negatively in the class of $2$-isometric
operators (see \cite{ACJS-1}). Originally formulated
for completely hyperexpansive operators (see
\cite[Question~2.11]{Ch-0}), the problem entails on
determining whether the Cauchy dual operator of a
member of the underlying class is subnormal. It is
worth mentioning that there are relatively broad
subclasses of $2$-isometric or $2$-hyperexpansive
operators for which this problem has an affirmative
solution (see \cite{ACJS-1} and \cite{B-S}
respectively). In the present paper we show that the
Cauchy dual subnormality problem has a negative
solution even in the class of cyclic $2$-isometric
operators. The counterexample is implemented with the
help of a weighted composition operator on $L^2$ space
over a directed graph with a circuit (see
Theorem~\ref{kontr-cyc}).

Apart from Introduction, the paper consists of three
parts. The first one gives the theoretical background
on weighted composition operators on $L^2$ spaces
needed in this paper. In the next one, we construct a
concrete class of weighted composition operators on an
$L^2$ space and characterize the subnormality of their
Cauchy duals. In the last part, by specifying the
weights and using Hausdorff's moment problem technique
along with subtle classical analysis, we get the
required counterexample.

We will now provide the necessary concepts and facts
related to the issues discussed, placing more emphasis
on the Hausdorff moment problem. Given a complex
Hilbert space $\hh,$ we denote by $\ogr{\hh}$ the
$C^*$-algebra of all bounded linear operators on
$\hh.$ Let $T\in \ogr{\hh}.$ We write $|T|$ for
$(T^*T)^{1/2}$ and call it the {\em modulus} of $T.$
We say that $T$ is {\em cyclic} if there exists a
vector $e_0,$ called a {\em cyclic} vector of $T,$
such that the linear span of the set $\{T^n
e_0\}_{n=0}^{\infty}$ is dense in $\hh$. We call $T$
{\em subnormal} if there exist a complex Hilbert space
$\mathcal K$ and a normal operator $N\in \ogr{\mathcal
K}$ such that $\hh \subseteq \mathcal K$ (an isometric
embedding) and $Th=Nh$ for all $h\in \hh.$ If $T$ is
left-invertible (or equivalently $T$ is bounded from
below), then $T^*T$ is an invertible element of
$\ogr{\hh}$ and the operator $T':=T(T^*T)^{-1}$ is
called the {\em Cauchy dual operator} of $T$
(abbreviated to: the {\em Cauchy dual} of $T$).
Finally, $T$ is said to be a {\em $2$-isometry} (or
that $T$ is {\em $2$-isometric}) if
   \begin{align*}
I - 2 T^*T + T^{*2}T^2 = 0.
   \end{align*}
We refer the reader to \cite{Co91}, \cite{Ag-St} and
\cite{Sh} for more information on subnormal operators,
$2$-isometric operators and the Cauchy dual operation,
respectively.

Hereafter $\cbb$ stands for the field of complex
numbers. The ring of all polynomials in one complex
variable $z$ with complex coefficients is denoted by
$\cbb[z].$ Given a sequence
$\{\gamma_n\}_{n=0}^{\infty}$ of complex numbers, we
say that {\em $\gamma_n$ is a polynomial in $n$} (of
degree $d$) if there exists $p\in \cbb[z]$ (of degree
$d$) such that $\gamma_n=p(n)$ for all $n\in \zbb_+,$
where $\zbb_+:=\{0,1,2,\ldots\}.$ The following fact
will be used later (cf.\ \cite[Exercise~7.2]{Dick}).
   \begin{align} \label{delt-dela}
   \begin{minipage}{70ex}
{\em If $\gamma_n$ is a polynomial in $n$ of degree
$d,$ then}
   $$
\sum_{n=0}^m (-1)^n \binom{m}{n} \gamma_n = 0, \quad
m\Ge \max\{d+1,0\}.
   $$
   \end{minipage}
   \end{align}
A sequence $\{\gamma_n\}_{n=0}^{\infty}$ of real
numbers is called a {\em Stieltjes moment sequence}
(resp.\ {\em Hausdorff moment sequence}) if there
exists a positive Borel measure $\mu$ on $[0,\infty)$
(resp.\ $[0,1]$) such that
   \begin{align*}
\gamma_n = \int t^n d \mu(t), \quad n \in \zbb_+.
   \end{align*}
Clearly, any Hausdorff moment sequence is a Stieltjes
moment sequence, but not conversely. Using Lebesgue's
monotone convergence theorem, one can show the
following.
   \begin{align} \label{trut-uu}
   \begin{minipage}{75ex}
{\em Any bounded Stieltjes moment sequence is a
Hausdorff moment sequence.}
   \end{minipage}
   \end{align}
Recall that Hausdorff moment sequences can be
characterized as follows (see
\cite[Proposition~6.11]{B-C-R}).
   \begin{align} \label{haus-m-t}
   \begin{minipage}{75ex}
{\em A sequence $\gamma=\{\gamma_n\}_{n=0}^{\infty}$
of real numbers is a Hausdorff moment sequence if and
only if}
   $$
\sum_{n=0}^{m} (-1)^n \binom{m}{n} \gamma_{n+j} \Ge 0,
\quad j, m \in \zbb_+.
   $$
   \end{minipage}
   \end{align}
In this paper, we adhere to the conventions:
   \begin{align} \label{conv-3}
   \begin{minipage}{75ex}
$\sum_{j=m}^{n} a_j = 0$ and $\prod_{l=m}^{n} a_l =1$
whenever $m > n$ and whatever $a_j$'s are.
   \end{minipage}
   \end{align}
To simplify the notation, we write
   \begin{align*}
\{f=0\}=\{x\in X\colon f(x)=0\} \quad \text{and} \quad
\{f\neq 0\}=\{x\in X\colon f(x)\neq 0\},
   \end{align*}
whenever $f$ is a $\cbb$-valued or a
$[0,\infty]$-valued function on a set $X.$
   \section{Weighted composition operators on $L^2$-spaces}
Given a measure space $(X, \ascr, \mu)$, we denote by
$L^2(\mu)$ the complex Hilbert space of all square
$\mu$-integrable $\ascr$-measurable complex functions
on $X$ endowed with the standard inner product.
   \begin{dfn}
Let $(X, \ascr, \mu)$ be a $\sigma$-finite measure
space, $\phi \colon X \rightarrow X$ be an
$\ascr$-measurable map and $\mathsf w \colon X
\rightarrow \cbb$ be an $\ascr$-measurable function.
By a {\em weighted composition operator} in $L^2(\mu)$
we mean a mapping $\cfw \colon L^2(\mu) \supseteq
\mathscr D(\cfw) \rightarrow L^2(\mu)$ defined by
   \begin{align*}
\mathscr D(\cfw) & = \Big\{f \in L^2(\mu)\colon
\mathsf w \cdot (f \circ \phi) \in L^2(\mu) \Big\},
   \\
\cfw f &= \mathsf w \cdot (f \circ \phi), \quad f \in
\mathscr D(\cfw).
   \end{align*}
We call $\phi$ and $\mathsf w$ the {\em symbol} and
the {\em weight} of $\cfw,$ respectively.
\end{dfn}
As a matter of fact, $\cfw$ may not be well-defined.
The well-definiteness of $\cfw$ means that $\mathsf w
\cdot (f \circ \phi) = \mathsf w \cdot (g \circ \phi)$
a.e.\ $[\mu]$ whenever $f, g \colon X \rightarrow
\cbb$ are $\ascr$-measurable functions such that $f=g$
a.e.\ $[\mu]$, $f\in L^2(\mu)$ and $\mathsf w \cdot (f
\circ \phi) \in L^2(\mu)$. Below we recall several
basic properties of weighted composition operators
including well-definiteness (see
\cite[Proposition~7]{BJJS3}).
   \begin{pro} \label{skok-ki}
Let $(X, \ascr, \mu)$ be a $\sigma$-finite measure
space, $\phi \colon X \rightarrow X$ be an
$\ascr$-measurable map and $\mathsf w \colon X
\rightarrow \cbb$ be an $\ascr$-measurable function.
Then $\cfw$ is well defined if and only if
$\mu_{\mathsf w} \circ \phi^{-1} \ll \mu$, where
$\mu_{\mathsf w}$ and $\mu_{\mathsf w} \circ
\phi^{-1}$ are measures on $\ascr$ defined by
   \begin{align} \label{measure-1}
\mu_{\mathsf w} (\sigma) =\int_{\sigma} |\mathsf
w|^2d\mu \quad \text{and} \quad \mu_{\mathsf w} \circ
\phi^{-1} (\sigma) = \mu_{\mathsf
w}(\phi^{-1}(\sigma)) \quad \text{for } \sigma \in
\ascr.
   \end{align}
Moreover, if $\cfw$ is well defined, $\mathsf u\colon
X \to \cbb$ is $\ascr$-measurable and $\mathsf
u=\mathsf w$ a.e.\ $[\mu],$ then $C_{\phi,\mathsf u}$
is well defined and $C_{\phi,\mathsf u}=\cfw.$
   \end{pro}
To avoid repetition, we distinguish the following
assumption, which we will often refer to in this
article.
   \begin{align} \label{SA-1}
   \begin{minipage}{75ex}
$(X, \ascr, \mu)$ is a $\sigma$-finite measure space,
$\phi \colon X \rightarrow X$ is an $\ascr$-measurable
map and $\mathsf w \colon X \rightarrow \cbb$ is an
$\ascr$-measurable function such that $\mu_{\mathsf w}
\circ \phi^{-1} \ll \mu.$
   \end{minipage}
   \end{align}
If the condition \eqref{SA-1} holds, then by the
Radon-Nikodym theorem (see
\cite[Theorem~2.2.1]{Ash00}), there exists a unique
(up to a set of $\mu$-measure zero) $\ascr$-measurable
function $\hfw\colon X \rightarrow [0, \infty]$ such
that
   \begin{align} \label{measure-2}
\mu_{\mathsf w} \circ \phi^{-1}(\sigma) =
\int_{\sigma} \hfw d\mu, \quad \sigma \in \ascr.
   \end{align}

Before going further, we make an important
observation.
   \begin{lem} \label{Sam-2raz}
Suppose that \eqref{SA-1} holds. If $\hfw > 0$ a.e.\
$[\mu],$ then
   \begin{align} \label{skok-ki2}
\mu(\{\mathsf w\neq 0\} \cap \{\hfw \circ \phi =
0\})=0.
   \end{align}
   \end{lem}
   \begin{proof}
Since $\mu(\{\hfw = 0\})=0$ and $\{\hfw \circ \phi =
0\} = \phi^{-1}(\{\hfw = 0\}),$ we infer from
\eqref{SA-1} that $\mu_{\mathsf w}(\{\hfw \circ \phi =
0\})=0,$ which yields \eqref{skok-ki2}.
   \end{proof}
   In view of \cite[Proposition 8(v) and
Theorem~18]{BJJS3}, the following is valid.
   \begin{align} \label{bod-ed}
   \begin{minipage}{70ex}
{\em If \eqref{SA-1} holds, then $\cfw\in
\ogr{L^2(\mu)}$ if and only if $\hfw\in
L^{\infty}(\mu);$ if this is the case, then
$\|\cfw\|^2 = \|\hfw\|_{L^{\infty}(\mu)}$ and the
modulus $|\cfw|$ of $\cfw$ equals the operator
$M_{h^{1/2}_{\phi, \mathsf w}}$ of multiplication by
$h^{1/2}_{\phi, \mathsf w}$ in $L^2(\mu)$.}
   \end{minipage}
   \end{align}
From now on, we assume that $\hfw$ takes finite values
whenever $\cfw\in \ogr{L^2(\mu)}.$ This assumption is
justified by \eqref{bod-ed}.
   \begin{pro} \label{prop-erties}
Suppose that \eqref{SA-1} holds and $\cfw\in
\ogr{L^2(\mu)}.$ Then
   \begin{enumerate}
   \item[(i)] if $c\in (0,\infty)$, then
$\|\cfw f\| \Ge c \|f\|$ for all $f\in L^2(\mu)$ if
and only if $\hfw \Ge c^2$ a.e.\ $[\mu],$
   \item[(ii)] if $\cfw$ is bounded from below, then
the Cauchy dual $C'_{\phi, \mathsf w}$ of $\cfw$
equals $C_{\phi, \mathsf w'},$ where $\mathsf w'
\colon X \rightarrow \cbb$ is any $\ascr$-measurable
function such that
   \begin{align} \label{dual-w}
\mathsf w' =
   \begin{cases}
\frac{\mathsf w}{\hfw \circ \phi} & \text{on }
\{\mathsf w\neq 0\} \cap \{\hfw \circ \phi >0\},
   \\
0 & \text{on } \{\mathsf w = 0\};
   \end{cases}
   \end{align}
in this case, $h_{\phi, \mathsf w'}=\frac{1}{\hfw}$
a.e.\ $[\mu]$.
   \end{enumerate}
   \end{pro}
   \begin{proof}
(i) Adapting the proof of \cite[Proposition 4]{BJJS1},
we get the statement (i).

(ii) Note that by \eqref{bod-ed} and (i), there exists
$c_1,c_2 \in (0,\infty)$ such that $c_1 \Le \hfw \Le
c_2$ a.e.\ $[\mu].$ By virtue of \eqref{bod-ed}, we
have
   \begin{align}  \label{bez-kon}
C'_{\phi, \mathsf w} = \cfw(|\cfw|^2)^{-1} = \cfw
M^{-1}_{\hfw} = \cfw M_{g},
   \end{align}
where $g\colon X \to [0,\infty)$ is an
$\ascr$-measurable function such that
$g=\frac{1}{\hfw}$ on the set $\{\hfw > 0\}.$ It is
easily seen that
   \begin{align*}
\mathsf w' \overset{\eqref{dual-w}}= \mathsf w \cdot
(g \circ \phi) \quad \text{on} \quad (\{\mathsf w\neq
0\} \cap \{\hfw \circ \phi >0\}) \cup \{\mathsf w=
0\}.
   \end{align*}
Hence, by Lemma~\ref{Sam-2raz}, $\mathsf w' = \mathsf
w \cdot (g\circ \phi)$ a.e.\ $[\mu].$
Proposition~\ref{skok-ki} yields
   \begin{align*}
C'_{\phi, \mathsf w}\overset{\eqref{bez-kon}}=C_{\phi,
\mathsf w \cdot (g\circ \phi)}=C_{\phi, \mathsf w'}.
   \end{align*}
Now applying \cite[Theorem 1.6.12]{Ash00} and
\eqref{measure-2}, we obtain
   \allowdisplaybreaks
   \begin{align*}
\mu_{\mathsf w'} \circ \phi^{-1}(\Delta) =
\mu_{\mathsf w \cdot (g \circ \phi)} \circ
\phi^{-1}(\Delta) & \overset{\eqref{measure-1}}=
\int_X (\chi_{\Delta} \circ \phi) (g^2\circ \phi)
d\mu_{\mathsf w}
   \\
&\hspace{.3ex} = \int_{\Delta} g^2 \hfw d\mu
   \\
&\hspace{.3ex} = \int_\Delta \frac{1}{\hfw} d\mu,
\quad \varDelta\in \ascr,
   \end{align*}
which shows that $h_{\phi, \mathsf w'}=\frac{1}{\hfw}$
a.e.\ $[\mu]$. This completes the proof.
   \end{proof}
Recall that (see \cite[Lemma 26]{BJJS3}) if
\eqref{SA-1} holds and $\cfw \in \ogr{L^2(\mu)},$ then
for any integer $n\Ge 1,$ the $n$th power $C^n_{\phi,
\mathsf w}$ of $\cfw$ equals $C_{\phi^n, \mathsf
w_{[n]}},$ where $\phi^n$ denotes the $n$-fold
composition of $\phi$ with itself ($\phi^0$ is the
identity map on $X$) and $\mathsf w_{[n]} \colon X
\rightarrow \cbb$ is the function given by
   \begin{align} \label{hat-w}
\text{$\mathsf w_{[0]} = 1$ and $\mathsf w_{[n]} =
\prod_{j=0}^{n-1} \mathsf w \circ \phi^j$ for $n \Ge
1$.}
   \end{align}
Moreover, $h_{\phi^0, \mathsf w_{[0]}}=1$ a.e.\
$[\mu]$ and the following recurrence formula holds:
   \begin{align}  \label{rec1}
h_{\phi^{n+1}, \mathsf w_{[n+1]}} & = \hfw \cdot
\efw\big(h_{\phi^n, \mathsf w_{[n]}}\big)\circ
\phi^{-1} \text{ a.e.\ $[\mu]$}, \quad n\in \zbb_+,
   \end{align}
where $\efw(f)$ stands for the conditional expectation
of an $\ascr$-measurable function $f\colon X \to
[0,\infty)$ with respect to the $\sigma$-algebra
$\phi^{-1}(\ascr)$ and the measure $\mu_{\mathsf w}$;
we refer the reader to \cite[Sect.~2.4]{BJJS3} for the
precise definitions of $\efw(f)$ and $\efw(f) \circ
\phi^{-1}.$ The above discussion and
Proposition~\ref{prop-erties} yield the following.
   \begin{pro} \label{power-dual}
Suppose that \eqref{SA-1} holds, $\cfw \in
\ogr{L^2(\mu)}$ and $\cfw$ is bounded from below. Then
the $n$th power $C'^n_{\phi, \mathsf w}$ of the Cauchy
dual $C'_{\phi, \mathsf w}$ of $\cfw$ equals
$C_{\phi^n, \mathsf w'_{[n]}},$ where $\mathsf
w'_{[n]} \colon X \rightarrow \cbb$ is the function
given by
   \begin{align} \label{hat-u-k}
\mathsf w'_{[0]}=1 \text{ and } \mathsf w'_{[n]} =
\prod_{j=0}^{n-1} \mathsf w' \circ \phi^j \text{ for }
n\Ge 1,
   \end{align}
with $\mathsf w'$ as in \eqref{dual-w}. Moreover,
$h_{\phi^0, \mathsf w'_{[0]}}=1$ a.e.\ $[\mu]$ and
   \begin{align*}
h_{\phi^{n+1}, \mathsf w'_{[n+1]}} = \frac{1}{\hfw} \,
\mathsf{E}_{\phi,\mathsf w'}(h_{\phi^{n}, \mathsf
w'_{[n]}}) \circ \phi^{-1} \text{ a.e.\ $[\mu]$},
\quad n \in \zbb_+.
   \end{align*}
   \end{pro}
The subnormality of the Cauchy dual $\cfw'$ of a
left-invertible bounded weighted composition operator
$\cfw$ can be characterized as follows (cf.\
\cite[Theorem~4.5]{J-K}).
   \begin{pro} \label{subnormal-dual}
Suppose that \eqref{SA-1} holds, $\cfw \in
\ogr{L^2(\mu)}$ and $\cfw$ is bounded from below. Then
   \begin{enumerate}
   \item[(i)] $C'_{\phi, \mathsf w}$ is
subnormal if and only if $\{h_{\phi^n, \mathsf
w'_{[n]}}(x)\}_{n=0}^{\infty}$ is a Stieltjes moment
sequence for $\mu$-a.e.\ $x \in X,$
   \item[(ii)] if $\|\cfw (f)\| \Ge \|f\|$ for all $f\in L^2(\mu)$
and $C'_{\phi, \mathsf w}$ is subnormal, then
$\{h_{\phi^n, \mathsf w'_{[n]}}(x)\}_{n=0}^{\infty}$
is a Hausdorff moment sequence for $\mu$-a.e.\ $x \in
X.$
   \end{enumerate}
   \end{pro}
   \begin{proof}
(i) Combine \cite[Theorem 49]{BJJS3} with
Proposition~\ref{power-dual}

(ii) Suppose now that $\|\cfw (f)\| \Ge \|f\|$ for all
$f\in L^2(\mu)$ and $\cfw'$ is subnormal. Then $\cfw'$
is a contraction and consequently so is its $n$th
power $\cfw^{\prime n}$ for every $n\in \zbb_+.$ In
view of \eqref{bod-ed} and
Proposition~\ref{power-dual}, the sequence
$\{h_{\phi^n, \mathsf w'_{[n]}}(x)\}_{n=0}^{\infty}$
is bounded by $1$ for $\mu$-a.e.\ $x \in X.$ Applying
(i) and \eqref{trut-uu} completes the proof.
   \end{proof}
We conclude this section by characterizing bounded
$2$-isometric weighted composition operators (see
\cite[Lemma~2.3]{Jab03} for the case of composition
operators).
   \begin{pro} \label{C-phi-2-iso}
Suppose that \eqref{SA-1} holds and $\cfw \in
\ogr{L^2(\mu)}.$ Then the following conditions are
equivalent{\em :}
\begin{enumerate}
   \item[(i)] $\cfw$ is a $2$-isometry,
   \item[(ii)] $1-2\hfw + h_{\phi^2, \mathsf w_{[2]}} = 0$
a.e.\ $[\mu],$
   \item[(iii)] $1-2\hfw + \hfw
\cdot \efw (\hfw) \circ \phi^{-1} = 0$ a.e.\ $[\mu].$
   \end{enumerate}
   \end{pro}
   \begin{proof}
Since $C^n_{\phi, \mathsf w}=C_{\phi^n, \mathsf
w_{[n]}},$ we infer from \eqref{bod-ed} that
$C^{*n}_{\phi, \mathsf w}C^n_{\phi, \mathsf
w}=M_{h_{\phi^n, \mathsf w_{[n]}}}$ for any $n\in
\zbb_+.$ Applying the definition of $2$-isometricity,
one can show that (i) and (ii) are equivalent. That
(ii) and (iii) are equivalent follows from
\eqref{rec1}.
   \end{proof}
   \section{A family of weighted composition operators on $\ell^2(\zbb_+)$}
In this section we concentrate on a family of weighted
composition operators coming from \cite[Example
42]{BJJS1} (see also \cite[Section 3.2]{BJJS2}).
   \begin{exa} \label{prz-g}
Denote by $\ascr$ the power set $2^{\zbb_+}$ of
$\zbb_+$ and by $\mu$ the counting measure on
$2^{\zbb_+}$. Clearly, all selfmaps of $\zbb_+$ and
complex functions on $\zbb_+$ are $\ascr$-measurable.
Note that $(\zbb_+, \ascr, \mu)$ is a $\sigma$-finite
measure space. For $n\in \zbb_+,$ we denote by $e_n$
the element of $L^2(\mu)$ given by
   \begin{align*}
e_n(m) =
   \begin{cases}
1 & \text{if } m=n,
   \\
0 & \text{otherwise.}
   \end{cases}
   \end{align*}
Clearly, $\{e_n\}_{n=0}^{\infty}$ is an orthonormal
basis of $L^2(\mu).$ Define the map $\phi\colon \zbb_+
\to \zbb_+$ by
   \begin{align*}
\phi(n) = \begin{cases} 0 & n=0,
   \\
n-1 & \text{otherwise}.
   \end{cases}
   \end{align*}
Let $\mathsf w \colon \zbb_+ \rightarrow \cbb$ be any
function. It is obvious that $\mu_{\mathsf w} \circ
\phi^{-1} \ll \mu$ and so $\cfw$ is well defined (cf.\
Proposition~\ref{skok-ki}). It follows from
\eqref{measure-2} that
   \begin{align} \label{hfw-1}
\hfw(n) = \sum_{j \in \phi^{-1}(\{n\})} |\mathsf
w(j)|^2, \quad n \in \zbb_+,
   \end{align}
which yields
   \begin{align} \label{isom-1}
\hfw(0) = \alpha_{\mathsf w} \text{ and } \hfw(n) =
|\mathsf w(n+1)|^2 \text{ for } n \geqslant 1,
   \end{align}
where
   \begin{align*}
\alpha_{\mathsf w} := |\mathsf w(0)|^2+ |\mathsf
w(1)|^2.
   \end{align*}
Combined with \eqref{bod-ed}, this implies that
   \begin{align} \label{ogr-norm}
   \begin{minipage}{75ex}
{\em $\cfw\in \ogr{L^2(\mu)}$ if and only if
$\sup_{n\Ge 0} |\mathsf w(n)| < \infty;$ if this is
the case, then}
   $$
\|\cfw\|^2=\max\Big\{\alpha_{\mathsf w}, \, \sup_{n
\Ge 2}|\mathsf w(n)|^2\Big\}.
   $$
   \end{minipage}
   \end{align}
Moreover, by Proposition~\ref{prop-erties}(i), $\cfw$
is bounded from below if and only if
   \begin{align} \label{min-ma}
\min\Big\{\alpha_{\mathsf w}, \, \inf_{n \Ge
2}|\mathsf w(n)|^2\Big\} > 0.
   \end{align}
Concerning the cyclicity of $\cfw$, one can make the
following observation.
   \begin{align} \label{cyc-0}
   \begin{minipage}{70ex}
{\em If $\cfw\in \ogr{L^2(\mu)}$ and $\mathsf w(n)\neq
0$ for all $n\Ge 1,$ then $\cfw$ is cyclic with the
cyclic vector $e_0.$}
   \end{minipage}
   \end{align}
This can be deduced from the equality \eqref{C-phi}
below which is a direct consequence of the definition.
   \begin{align} \label{C-phi}
\cfw e_n &= \begin{cases} \mathsf w(0)e_0 + \mathsf
w(1)e_1 & \text{ if } n=0,
   \\
\mathsf w(n+1)e_{n+1} & \text{ if } n \Ge 1.
   \end{cases}
   \end{align}
Using Proposition~\ref{prop-erties}, we can describe
the Cauchy dual $\cfw'$ of a left-invertible $\cfw$ as
follows.
   \begin{align} \label{prim-w}
   \begin{minipage}{75ex}
{\em If $\cfw \in \ogr{L^2(\mu)}$ and \eqref{min-ma}
holds, then $\cfw'=C_{\phi, \mathsf w'},$ where}
   $$
\mathsf w'(n) =
   \begin{cases}
\frac{\mathsf{w}(n)}{\alpha_{\mathsf w}} & \text{ if }
n=0, 1,
   \\[1ex]
\frac{1}{\overline{{\mathsf w}(n)}} & \text{ if } n
\Ge 2.
   \end{cases}
   $$
   \end{minipage}
   \end{align}
As a consequence, we get
   \begin{align*}
\cfw' e_n =
   \begin{cases}
\frac{\mathsf w(0)}{\alpha_{\mathsf w}}e_0 +
\frac{\mathsf w(1)}{\alpha_{\mathsf w}} e_1 & \text{
if }n=0,
   \\[1ex]
\frac{1}{\overline{{\mathsf w}(n+1)}} e_{n+1} & \text{
if } n \Ge 1.
   \end{cases}
   \end{align*}
   \hfill{$\diamondsuit$}
   \end{exa}
Below, we give necessary and sufficient conditions for
the weighted composition operator $\cfw$ from
Example~\ref{prz-g} to be $2$-isometric. Before we do
this, we define the functions $\xi_n\colon [1,\infty)
\to [1,\infty)$, where $n\in \zbb_+,$ by
   \begin{align} \label{xin}
\xi_n(x) = \sqrt{\frac{1+ (n+1)(x^2-1)}{1+ n(x^2-1)}},
\quad x \in [1,\infty), \, n\in \zbb_+.
   \end{align}
   \begin{pro} \label{2-iso-c}
Let $X$, $\ascr$, $\mu$, $\phi$ and $\mathsf w$ be as
in Example~{\em \ref{prz-g}}. Assume that $\cfw\in
\ogr{L^2(\mu)}.$ Then $\cfw$ is a $2$-isometry if and
only if
   \begin{align} \label{w2}
|\mathsf w(2)| \geq 1, \quad |\mathsf w(n+2) |=
\xi_n(|\mathsf w(2)|), \quad n \in \zbb_+,
   \end{align}
and
   \begin{align} \label{w3}
\begin{cases} |\mathsf w(0)|=1
& \text{ if } \mathsf w(1)=0,
   \\[1ex]
|\mathsf w(2)| = \frac{\sqrt{\alpha_{\mathsf
w}(2-|\mathsf w(0)|^2)-1}}{|\mathsf w(1)|} & \text{ if
} \mathsf w(1) \neq 0.
   \end{cases}
   \end{align}
Moreover, if $\cfw$ is a $2$-isometry, then
   \begin{enumerate}
   \item[(i)] if $\mathsf w(1)\neq 0$ and either
$|\mathsf w(0)|=1$ or $\alpha_{\mathsf w} = 1,$ then
$|\mathsf w(n)|=1$ for all $n \Ge 2,$
   \item[(ii)]
$(\alpha_{\mathsf w}-1)(1-|\mathsf w(0)|^2) \Ge 0.$
   \end{enumerate}
   \end{pro}
   \begin{proof}
Arguing as in \eqref{hfw-1} with $(\phi^2,\mathsf
w_{[2]})$ in place of $(\phi,\mathsf w)$ and using
\eqref{hat-w}, we obtain
   \begin{align*}
h_{\phi^{2}, \mathsf w_{[2]}} (n) =
   \begin{cases}
|\mathsf w(n+1)|^2 |\mathsf w(n+2)|^2 & \text{if } n\Ge 1,
   \\[1ex]
|\mathsf w(0)|^4 + |\mathsf w(0)|^2 |\mathsf w(1)|^2 +
|\mathsf w(1)|^2 |\mathsf w(2)|^2 & \text{if } n=0.
   \end{cases}
   \end{align*}
Combined with Proposition~\ref{C-phi-2-iso}(ii), we
deduce that $\cfw$ is a $2$-isometry if and only~if
   \begin{align} \label{wnu-2}
|\mathsf w(1)|^2|\mathsf w(2)|^2 = \alpha_{\mathsf
w}(2-|\mathsf w(0)|^2)-1
   \end{align}
and
   \begin{align} \label{wnu-1}
1 - 2 |\mathsf w(n+2)|^2 + |\mathsf w(n+2)|^2|\mathsf
w(n+3)|^2 = 0, \quad n\in \zbb_+.
   \end{align}
It is easily seen that \eqref{w3} is equivalent to
\eqref{wnu-2}. Observe that the condition
\eqref{wnu-1} holds if and only if the unilateral
weighted shift on $\ell^2(\zbb_+)$ with weights
$\{|w(n+2)|\}_{n=0}^{\infty}$ is $2$-isometric. Hence,
by \cite[Lemma~6.1]{Ja-St}, the condition
\eqref{wnu-1} is equivalent to \eqref{w2}. Using
$|\mathsf w(2)| \geq 1$ and the fact that the
expression under the sign of the square root in
\eqref{w3} is nonnegative, we get the ``moreover''
part.
   \end{proof}
It is well known that any $2$-isometry is expansive
(see \cite[Lemma~1]{R-0}), so we can consider its
Cauchy dual operator. Below, we follow the conventions
\eqref{conv-3}.
   \begin{thm} \label{sub-cauch}
Let $X$, $\ascr$, $\mu$, $\phi$ and $\mathsf w$ be as
in Example~{\em \ref{prz-g}}. Assume that $\cfw\in
\ogr{L^2(\mu)}$ and $\cfw$ is a $2$-isometry. Then the
following conditions hold{\em :}
   \begin{enumerate}
   \item[(i)] the Cauchy dual $C'_{\phi, \mathsf w}$ of $\cfw$ is
subnormal if and only if the sequence $\{h_{\phi^{n},
\mathsf w'_{[n]}}(0)\}_{n=0}^{\infty}$ is a Hausdorff
moment sequence,
   \item[(ii)] the value of $h_{\phi^{n}, \mathsf w'_{[n]}}$ at
$0$ is given by the following explicit formula
   \begin{align} \label{formula-1}
h_{\phi^{n}, \mathsf w'_{[n]}}(0) = \frac{|\mathsf
w(0)|^{2n}}{\alpha_{\mathsf w}^{2n}} +
\sum_{j=0}^{n-1} \frac{|\mathsf
w(0)|^{2(n-j-1)}|\mathsf w(1)|^2}{\alpha_{\mathsf
w}^{2(n-j)}({1+j(|\mathsf w(2)|^2-1))}}, \quad n\in
\zbb_+.
   \end{align}
   \end{enumerate}
   \end{thm}
   \begin{proof}
(i) The ``only if'' part is a direct consequence
Proposition~\ref{subnormal-dual}(ii). In view of
Proposition~\ref{subnormal-dual}(i) to prove the
''if'' part, it suffices to show that the sequence
$\{h_{\phi^{n}, \mathsf w'_{[n]}}(k)\}_{n=0}^{\infty}$
is a Stieltjes moment sequence for every $k\Ge 1.$ Fix
$k\Ge 1.$ Arguing as in \eqref{hfw-1} with $(\phi^{n},
\mathsf w'_{[n]})$ in place of $(\phi,\mathsf w),$ we
verify that
   \begin{align*}
h_{\phi^{n}, \mathsf w'_{[n]}}(k) = |\mathsf
w'_{[n]}(k + n)|^2 \overset{\eqref{hat-w}} =
\prod_{j=1}^{n} |\mathsf w'(k+j)|^2
\overset{\eqref{prim-w}}= \frac{1}{\prod_{j=1}^{n}
|\mathsf w(k+j)|^2}, \quad n\Ge 1.
   \end{align*}
Since the unilateral weighted shift on
$\ell^2(\zbb_+)$ with weights $\{|\mathsf
w(k+1+n)|\}_{n=0}^{\infty}$ is $2$-isometric (see the
proof of Proposition~\ref{2-iso-c}), we deduce
from\footnote{Recall that a $2$-isometry is
$m$-isometric for every integer $m\Ge 2$ (see
\cite[Paper I, \S 1]{Ag-St}), and thus by
\cite[Lemma~1(a)]{R-0} it is completely
hyperexpansive. } \cite[Remark~4]{At} that the
unilateral weighted shift on $\ell^2(\zbb_+)$ with
weights $\big\{\frac{1}{|\mathsf
w(k+1+n)|}\big\}_{n=0}^{\infty}$ is subnormal, which
by Berger-Gellar-Wallen theorem (see \cite{GW,Hal}) is
equivalent to the fact that the sequence
$\{h_{\phi^{n}, \mathsf w'_{[n]}}(k)\}_{n=0}^{\infty}$
is a Stieltjes moment sequence.

(ii) Arguing as in (i), we get
   \allowdisplaybreaks
   \begin{align*}
h_{\phi^{n}, \mathsf w'_{[n]}}(0) & = \sum_{j= 0}^{n}
|\mathsf w'_{[n]}(j)|^2
   \\
&\hspace{-.7ex} \overset{\eqref{hat-u-k}}= \sum_{j=
0}^{n} \; \prod_{l=0}^{j-1} |\mathsf
w'(\phi^{l}(j))|^2 \prod_{l=j}^{n-1} |\mathsf
w'(\phi^l(j))|^2
   \\
& = |\mathsf w'(0)|^{2n} + \sum_{j= 1}^{n} |\mathsf
w'(0)|^{2(n-j)} \prod_{l=1}^{j} |\mathsf w'(l)|^2
         \\
& \hspace{-.7ex} \overset{\eqref{prim-w}} =
\frac{|\mathsf w(0)|^{2n}}{\alpha_{\mathsf w}^{2n}} +
\frac{|\mathsf w(0)|^{2(n-1)} |\mathsf
w(1)|^2}{\alpha_{\mathsf w}^{2n}} + \sum_{j= 2}^{n}
\frac{|\mathsf w(0)|^{2(n-j)}|\mathsf
w(1)|^2}{\alpha_{\mathsf w}^{2(n-j+1)}\prod_{l=2}^{j}
|\mathsf w(l)|^2}
   \\
& \hspace{-.7ex} \overset{\eqref{w2}}= \frac{|\mathsf
w(0)|^{2n}}{\alpha_{\mathsf w}^{2n}} + \frac{|\mathsf
w(0)|^{2(n-1)} |\mathsf w(1)|^2}{\alpha_{\mathsf
w}^{2n}} + \sum_{j= 1}^{n-1} \frac{|\mathsf
w(0)|^{2(n-j-1)}|\mathsf w(1)|^2}{\alpha_{\mathsf
w}^{2(n-j)}(1+j(|\mathsf w(2)|^2 -1))}
   \\
& = \frac{|\mathsf w(0)|^{2n}}{\alpha_{\mathsf
w}^{2n}} + \sum_{j= 0}^{n-1} \frac{|\mathsf
w(0)|^{2(n-j-1)}|\mathsf w(1)|^2}{\alpha_{\mathsf
w}^{2(n-j)}(1+j(|\mathsf w(2)|^2 -1))}, \quad n\Ge 2.
   \end{align*}
   It is a matter of simple verification that
\eqref{formula-1} holds for $n=0,1$ as well.
   \end{proof}
   \begin{cor}
Let $X$, $\ascr$, $\mu$, $\phi$ and $\mathsf w$ be as
in Example~{\em \ref{prz-g}}. Assume that $|\mathsf
w(0)|=|\mathsf w(n)|=1$ for every integer $n \ge 2.$
Then $\cfw\in \ogr{L^2(\mu)},$ $\cfw$ is a
$2$-isometry and $\cfw'$ is subnormal. Moreover,
$\cfw'$ is an isometry if and only if $\cfw$ is an
isometry or, equivalently, if and only if $\mathsf
w(1)=0.$
   \end{cor}
   \begin{proof}  By \eqref{ogr-norm} and
Proposition~\ref{2-iso-c}, $\cfw\in \ogr{L^2(\mu)}$
and $\cfw$ is a $2$-isometry. Using
Theorem~\ref{sub-cauch}(ii), we see that if $\mathsf
w(1) \neq 0,$ then
   \allowdisplaybreaks
   \begin{align*}
h_{\phi^{n}, \mathsf w'_{[n]}}(0) &=
\frac{1}{\alpha_{\mathsf w}^{2n}} + \sum_{j=0}^{n-1}
\frac{|\mathsf w(1)|^2}{\alpha_{\mathsf w}^{2(n-j)}}
   \\
& = \frac{1}{\alpha_{\mathsf w}^{2n}} + \frac{|\mathsf
w(1)|^2(\alpha_{\mathsf w}^{2n}-1)}{\alpha_{\mathsf
w}^{2n}(\alpha_{\mathsf w}^{2}-1)}
   \\
& =\frac{|\mathsf w(1)|^2}{\alpha_{\mathsf w}^{2}-1} +
\frac{\alpha_{\mathsf w}}{\alpha_{\mathsf w}+1}
(\alpha_{\mathsf w}^{-2})^n
   \\
& =\frac{1}{2+|\mathsf w(1)|^2} + \frac{1+|\mathsf
w(1)|^2}{2+|\mathsf w(1)|^2} (\alpha_{\mathsf
w}^{-2})^n , \quad n\in \zbb_+,
   \end{align*}
which implies that the sequence $\{h_{\phi^{n},
\mathsf w'_{[n]}}(0)\}_{n=0}^{\infty}$ is a Hausdorff
moment sequence. The same is true if $\mathsf w(1)=0$
because then by \eqref{formula-1}, $h_{\phi^{n},
\mathsf w'_{[n]}}(0)=1$ for all $n\in \zbb_+.$
Applying Theorem~\ref{sub-cauch}(i), we see that
$\cfw'$ is subnormal. The first equivalence in the
``moreover'' part is true for arbitrary
left-invertible operators, while the second is a
direct consequence of \cite[(2.22)]{BJJS3} and
\eqref{isom-1}. This completes the proof.
   \end{proof}
   \section{Main example}
The following example shows that the Cauchy dual
subnormality problem has a negative solution even in
the class of cyclic operators.
   \begin{exa} \label{exm1.6} (Example~\ref{prz-g}
continued). Let $X$, $\ascr$, $\mu$, $\phi$ and
$\mathsf w$ be as in Example~\ref{prz-g}. To achieve
the main purpose of this paper, we will begin by
specifying the weight $\mathsf w.$ Let $x$ is any
positive real number and let $\mathsf w$ be the weight
function constructed as follows (for notational
convenience the dependence of $\mathsf w$ on $x$ will
not be expressed explicitly). Set
   \begin{align}  \label{wydz-w2}
\mathsf w(0)=\frac{1}{\sqrt{2}}, \quad \mathsf
w(1)=\sqrt{\frac{1}{2} + x}.
   \end{align}
Then clearly $\alpha_{\mathsf w}(2-\mathsf w(0)^2)-1
> 0$ and
   \begin{align} \label{wydz-w3}
\mathsf w(2):= \frac{\sqrt{\alpha_{\mathsf
w}(2-\mathsf w(0)^2)-1}}{\mathsf w(1)} =
\sqrt{\frac{1+3x}{1+2x}} > 1.
   \end{align}
Set
   \begin{align} \label{wydz-w4}
\mathsf w(n+2):= \xi_n(\mathsf w(2)), \quad n \Ge 1,
   \end{align}
where the functions $\xi_n$ are as in \eqref{xin}. It
follows from the definition of $\mathsf w$ and
\eqref{xin} that $\sup_{n\Ge 0} \mathsf w(n) <
\infty,$ so by \eqref{ogr-norm}, $\cfw\in
\ogr{L^2(\mu)}.$ Using Proposition~\ref{2-iso-c}, we
deduce that $\cfw$ is a $2$-isometry. It follows from
\eqref{formula-1}~ that
   \begin{align*}
h_{\phi^{n}, \mathsf w'_{[n]}}(0) =
\frac{1}{2^n(1+x)^{2n}} + \sum_{j=0}^{n-1}
\frac{1+2x}{2^{n-j}(1+x)^{2(n-j)}(1+j(\mathsf
w(2)^2-1))}, \quad n\in \zbb_+.
   \end{align*}
By \eqref{wydz-w3}, we have
   \begin{align*}
1+j(\mathsf w(2)^2-1) = \frac{1+(j+2)x}{1+2x},
   \end{align*}
which yields
   \begin{align} \label{now-hi}
h_{\phi^{n}, \mathsf w'_{[n]}}(0) =
\frac{1}{2^n(1+x)^{2n}} \bigg( 1 + (1+2x)^2
\sum_{j=0}^{n-1} \frac{2^{j}(1+x)^{2j}}{1+(j+2)x}
\bigg), \quad n\in \zbb_+.
   \end{align}
   \hfill{$\diamondsuit$}
   \end{exa}
Below we will continue the necessary preparations to
achieve the main goal of this paper. For each $n\in
\zbb_+$, we define the real valued function $\omega_n$
on $\varOmega_n:= (-\frac{1}{n+1},\infty)$ by
   \begin{align*}
\omega_n(x) = \frac{1}{2^n(1+x)^{2n}} \bigg( 1 +
(1+2x)^2 \sum_{j=0}^{n-1}
\frac{2^{j}(1+x)^{2j}}{1+(j+2)x} \bigg), \quad x \in
\varOmega_n, \, n \in \zbb_+.
   \end{align*}
Clearly, the following holds
   \begin{align} \label{omg-ha}
\omega_{n}(x) = \frac{1}{2^n(1+x)^{2n}} \left( 1 +
(1+2x)^2 S_n(x) \right), \quad \quad x \in
\varOmega_n, \, n\in \zbb_+,
   \end{align}
where
   \begin{align*}
S_n(x) := \sum_{j=0}^{n-1}
\frac{2^{j}(1+x)^{2j}}{1+(j+2)x}, \quad x \in
\varOmega_n,\, n\in \zbb_+.
   \end{align*}
Set
   \begin{align} \label{D-m-epsilon-0}
D_m(x) = \sum_{n=0}^m (-1)^n \binom{m}{n} \omega_n(x),
\quad x \in \varOmega_m, \, m \in \zbb_+.
   \end{align}
Then
   \begin{align} \label{D-m-epsilon}
D^{(l)}_m(x) = \sum_{n=0}^m (-1)^n \binom{m}{n}
\omega^{(l)}_n(x), \quad x \in \varOmega_m, \, m \in
\zbb_+, \, l \in \zbb_+,
   \end{align}
where $D^{(l)}_m$ (resp.\ $\omega^{(l)}_n$) stands for
the $l$-th derivative of $D_m$ (resp.\ $\omega_n$).
Applying the general Leibniz rule, we see that the
$l$-th derivative $S_n^{(l)}$ of $S_n$ is given by
   \begin{align*}
S_n^{(l)}(x) = \sum_{j=0}^{n-1} {2^{j}} \sum_{k=0}^{l}
\binom{l}{k} \Big(\frac{1}{1+(j+2)x}\Big)^{(k)}
((1+x)^{2j})^{(l-k)}, \quad x \in \varOmega_n, \, l, n
\in \zbb_+.
   \end{align*}
In particular, for every $n\in \zbb_+$ we have
   \allowdisplaybreaks
   \begin{align}  \notag
S_n(0) &= 2^n-1,
   \\  \notag
S_n^{(1)}(0) &= n 2^n - 4(2^n-1),
   \\ \notag
S_n^{(2)}(0) &= 2n^2 2^n -10n 2^n+ 24(2^n-1),
   \\ \notag
S_n^{(3)}(0) &= 2 n^3 2^n - 30 n^2 2^n +100 n 2^n -
192 (2^n -1),
   \\  \label{9}
S_n^{(4)}(0) &= 8 n^4 2^n - 56 n^3 2^n + 460 n^2 2^n -
1324 n 2^n + 2208 (2^n -1).
   \end{align}
Now we compute the $l$th derivative $\omega^{(l)}_n$
of $\omega_n.$
   \begin{lem} \label{lem-1.6}
The following conditions are valid{\em :}
   \begin{enumerate}
   \item[(i)] if $n\in \zbb_+,$
$l$ is a positive integer and $x$ varies over
$\varOmega_n,$ then $($see \eqref{conv-3}$)$
   \allowdisplaybreaks
   \begin{align}  \notag
& \omega^{(l)}_n(x) = \Big( \frac{1}{(1+x)^{2n}}
\Big)^{(l)} \left[\frac{1 + (1+2x)^2 S_n(x)
}{2^n}\right]
   \\ \notag
&\hspace{4ex} + l \Big( \frac{1}{(1+x)^{2n}}
\Big)^{(l-1)} \left[ \frac{2^2(1+2x) S_n(x) + (1+2x)^2
S_n^{(1)}(x)}{2^n} \right]
   \\ \notag
&\hspace{4ex}+ \sum_{i = 2}^l \binom{l}{i} \Big(
\frac{1}{(1+x)^{2n}} \Big)^{(l-i)}
   \\  \label{zus-as}
& \hspace{10ex}\times \left[\frac{2^3 \binom{i}{2}
S_n^{(i-2)}(x) + 2^2 \binom{i}{1} (1+2x)
S_n^{(i-1)}(x) + (1+2x)^2 S_n^{(i)}(x)}{2^n} \right],
   \end{align}
   \item[(ii)] if $l \in \{ 0, 1, 2, 3\}$ and $m\in \{4,5,6,\ldots\},$
then $D_m^{(l)}(0)=0.$
   \end{enumerate}
   \end{lem}
   \begin{proof}
(i) Applying the general Leibniz rule twice, we get
   \allowdisplaybreaks
   \begin{align*}
2^n \omega^{(l)}_n(x) & \overset{\eqref{omg-ha}}=
\sum_{i=0}^{l} \binom{l}{i} \Big( \frac{1}{(1+x)^{2n}}
\Big)^{(l-i)} \Big( 1 + (1+2x)^2 S_n(x) \Big)^{(i)}
   \\
&= \Big( \frac{1}{(1+x)^{2n}} \Big)^{(l)} \Big( 1 +
(1+2x)^2 S_n(x) \Big)
   \\
&\hspace{5ex} + \sum_{i=1}^{l} \binom{l}{i} \Big(
\frac{1}{(1+x)^{2n}} \Big)^{(l-i)} \Big((1+2x)^2
S_n(x) \Big)^{(i)}
   \\
&= \Big( \frac{1}{(1+x)^{2n}} \Big)^{(l)} \Big( 1 +
(1+2x)^2 S_n(x) \Big)
   \\
&\hspace{5ex} + l \Big( \frac{1}{(1+x)^{2n}}
\Big)^{(l-1)}\Big( 2^2(1+2x) S_n(x) + (1+2x)^2
S_n^{(1)}(x) \Big)
   \\
&\hspace{5ex} + \sum_{i=2}^{l} \binom{l}{i} \Big(
\frac{1}{(1+x)^{2n}} \Big)^{(l-i)} \sum_{m=0}^{i}
\binom{i}{m} ((1+2x)^2)^{(m)} S_n^{(i-m)}(x),
   \end{align*}
which implies that \eqref{zus-as} holds for $l\Ge 2,$
$n\in \zbb_+$ and $x\in \varOmega_n.$ It is a matter
of routine to verify that \eqref{zus-as} holds for
$l=1,$ $n\in \zbb_+$ and $x\in \varOmega_n$ as well.
This yields~(i).

(ii) Using \eqref{9} we verify that the factors in the
square brackets appearing in \eqref{zus-as} (the third
one for $i=2,3$) when calculated at $0$ are of the
form:
   \allowdisplaybreaks
   \begin{gather} \notag
\frac{S_n(0) +1}{2^n} = p_0(n), \text{ where $p_0\in
\cbb[z]$ is of degree $0$},
   \\ \notag
\frac{4 S_n(0) + S_n^{(1)}(0)}{2^n} = p_1(n), \text{
where $p_1\in \cbb[z]$ is of degree $1$},
   \\ \notag
\frac{8S_n(0) + 8S^{(1)}_n(0) + S_n^{(2)}(0)}{2^n} =
p_2(n), \text{ where $p_2\in \cbb[z]$ is of degree
$2$},
   \\ \label{do-10}
\frac{24 S^{(1)}_n(0) + 12 S^{(2)}_n(0) +
S_n^{(3)}(0)}{2^n} = p_3(n), \text{ where $p_3\in
\cbb[z]$ is of degree $3$}.
   \end{gather}
This together with \eqref{zus-as} implies that
$\omega^{(l)}_n(0)$ is a polynomial in $n$ of degree
at most $l$ whenever $l\in \{0,1,2,3\}.$ Therefore,
(ii) is a direct consequence of \eqref{delt-dela} and
\eqref{D-m-epsilon}.
   \end{proof}
   \begin{figure}[t]
   \centering
\includegraphics[width=0.9\textwidth]{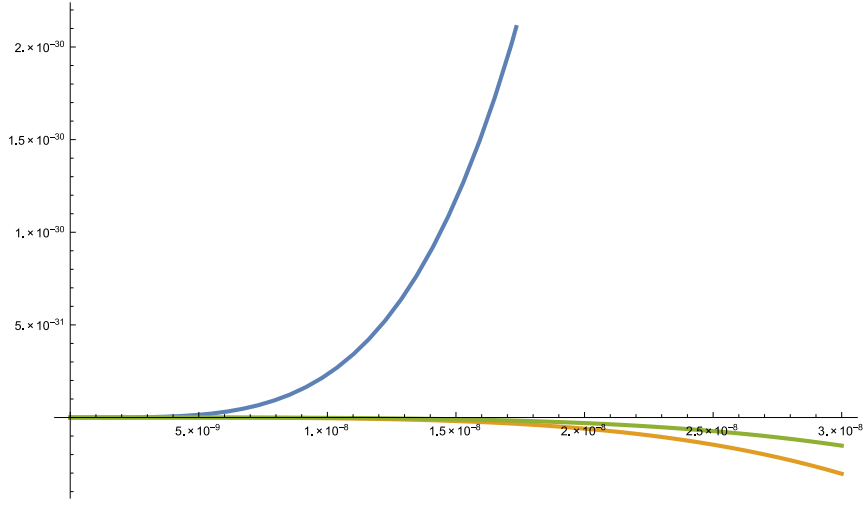}
\caption{Plots of $D_m(x)$ for $m = 4, 5,6$ showing
that $D_m(x)$ takes negative values in a neighborhood
of $x=0$ for $m = 5, 6$ and remains nonnegative for $m
= 4$.} \label{fig1}
   \end{figure}
   \begin{lem} \label{kontr-cyc0}
For every integer $m\Ge 5,$ there exists
$\varepsilon_m \in (0,\infty)$ such that $D_m(x) < 0$
for every $x\in (0,\varepsilon_m).$
   \end{lem}
   \begin{proof}
It follows from \eqref{9} that that the third factor
in the square brackets appearing in \eqref{zus-as} for
$i=4$ when calculated at $0$ is of the form
   \begin{align*}
\frac{48 S_n^{(2)}(0) + 16 S_n^{(3)}(0) +
S_n^{(4)}(0)}{2^n} = p_4(n)-\frac{288}{2^n}, \quad
x\in \varOmega_n, \, n\in \zbb_+,
   \end{align*}
where $p_4\in \cbb[z]$ is of degree $4.$ Combined with
\eqref{do-10} and \eqref{zus-as}, this yields
   \allowdisplaybreaks
   \begin{align*} &
\omega^{(4)}_n(0) = p_0(n) \left( \frac{1}{(1+x)^{2n}}
\right)^{(4)}\Big|_{x=0} + 4 p_1(n) \left(
\frac{1}{(1+x)^{2n}} \right)^{(3)}\Big|_{x=0}
      \\
& \hspace{10ex}+ 6 p_2(n) \left( \frac{1}{(1+x)^{2n}}
\right)^{(2)} \Big|_{x=0} + 4 p_3(n) \left(
\frac{1}{(1+x)^{2n}} \right)^{(1)}\Big|_{x=0}
   \\
&\hspace{10ex} + \Big(p_4(n)-\frac{288}{2^n}\Big)
\Big(\frac{1}{(1+x)^{2n}} \Big)^{(0)}\Big|_{x=0},
\quad n\in \zbb_+.
   \end{align*}
Hence, $\omega^{(4)}_n(0) + \frac{288}{2^n}$ is a polynomial in $n$
of degree at most $4.$ This fact, together with
\eqref{delt-dela} and \eqref{D-m-epsilon} implies that
   \begin{align*}
D_m^{(4)}(0) =-288 \sum_{n=0}^{m} (-1)^{n}
\binom{m}{n} 2^{-n}= - \frac{288}{2^m} < 0, \quad m
\ge 5.
   \end{align*}
Now applying Lemma~\ref{lem-1.6}(ii) and Taylor's
theorem to $D_m$ (see \cite[Theorem~5.15]{Rud76} with
$n=4$ and $\alpha=0$) completes the proof.
   \end{proof}
Concerning Lemma~ \ref{kontr-cyc0}, the reader is
referred to Figure~\ref{fig1}. Now we are ready to
state the main result of the paper.
   \begin{thm} \label{kontr-cyc}
Let $X$, $\ascr$, $\mu$, $\phi$ and $\mathsf w$ be as
in Example~{\em \ref{exm1.6}} $($with $\mathsf w$
given by \eqref{wydz-w2}-\eqref{wydz-w4}$).$ Then
there exists $\varepsilon \in (0,\infty)$ such that
for every $x \in (0, \varepsilon),$ $\cfw\in
\ogr{L^2(\mu)}$ and $\cfw$ is a cyclic $2$-isometry
such that $\cfw'$ is not subnormal.
   \end{thm}
   \begin{proof}
It follows from \eqref{cyc-0} and Example~\ref{exm1.6}
that if $x\in (0,\infty),$ then $\cfw\in
\ogr{L^2(\mu)}$ and $\cfw$ is a cyclic $2$-isometry.
In turn, by Lemma~\ref{kontr-cyc0}, there exists
$\varepsilon \in (0,\infty)$ such that $D_5(x) < 0$
for every $x \in (0, \varepsilon).$ Applying
\eqref{haus-m-t}, \eqref{D-m-epsilon-0} and
\eqref{now-hi}, we deduce that whenever $x \in (0,
\varepsilon),$ $\{h_{\phi^{n}, \mathsf
w'_{[n]}}(0)\}_{n=0}^{\infty}$ is not a Hausdorff
moment sequence and consequently, by
Theorem~\ref{sub-cauch}(i), $\cfw'$ is not subnormal.
   \end{proof}
We conclude the paper with the following observation.
   \begin{rem} \label{conc-rem}
It is worth noting that if the parameter $x$ in
Example~\ref{exm1.6} equals $0$, then $\cfw\in
\ogr{L^2(\mu)}$ and, by \eqref{isom-1},
\eqref{wydz-w2}, \eqref{wydz-w3} and \eqref{wydz-w4}
with $x=0$, $\hfw(n) = 1$ for all $n\in \zbb_+,$ which
implies that $\cfw$ is an isometry\footnote{\;In fact,
in view of \eqref{C-phi},
$\big(\cfw(L^2(\mu))\big)^{\perp} = \cbb \cdot (e_{0}
- e_{1}),$ so $\cfw$ is not unitary.} (see
\cite[(2.22)]{BJJS3}). It follows that $\cfw'$ is an
isometry and, as such, is subnormal (because each
isometry has a unitary extension possibly in a larger
Hilbert space, see e.g.,
\cite[Proposition~I.2.3]{SF70}). Therefore,
Theorem~\ref{kontr-cyc} breaks down at the point
$x=0.$
   \end{rem}
   

\begin{thebibliography}{1}
   \bibitem{Ag-St} J. Agler, M. Stankus, $m$-isometric
transformations of Hilbert spaces, I, II, III, {\it
Integr. Equ. Oper. Theory} {\bf 21, 23, 24} (1995,
1995, 1996), 383-429, 1-48, 379-421.
   \bibitem{ACJS-1}  A. Anand, S. Chavan, Z. J. Jab{\l}o\'nski,
J. Stochel, A solution to the Cauchy dual subnormality
problem for 2-isometries, {\em J. Funct. Anal.} {\bf
277} (2019), 108292, 51 pp.
   \bibitem{Ash00} R. B. Ash, {\em Probability and measure theory},
Harcourt/Academic Press, Burlington, 2000.
   \bibitem{At} A. Athavale, On completely hyperexpansive operators,
{\em Proc. Amer. Math. Soc.} {\bf 124} (1996),
3745-3752.
   \bibitem{B-S} C. Badea, L. Suciu, The Cauchy dual and
$2$-isometric liftings of concave operators, {\em J.
Math. Anal. Appl.} {\bf 472} (2019), 1458-1474.
   \bibitem{B-C-R} C. Berg, J. P. R. Christensen, P. Ressel,
{\em Harmonic Analysis on Semigroups},
Springer-Verlag, Berlin 1984.
   \bibitem{BJJS1} P. Budzy\'{n}ski, Z. J.
Jab{\l}o\'nski, I. B. Jung, J. Stochel, Unbounded
subnormal composition operators in $L_2$-spaces., {\em
J. Funct. Anal.} {\bf 269} (2015), 2110-2164.
   \bibitem{BJJS2} P. Budzy\'{n}ski, Z. J.
Jab{\l}o\'nski, I. B. Jung, J. Stochel, Subnormality
of unbounded composition operators over one-circuit
directed graphs: exotic examples, {\em Adv. Math.}
{\bf 310} (2017), 484-556.
   \bibitem{BJJS3} P. Budzy\'{n}ski, Z. J.
Jab{\l}o\'nski, I. B. Jung, J. Stochel, {\em Unbounded
weighted composition operators in $L^2$-spaces,} Lect.
Notes Math., Volume 2209, Springer 2018, xii+180 pp.
   \bibitem{Ch-0}  S. Chavan, On operators Cauchy
dual to $2$-hyperexpansive operators, {\em Proc. Edin.
Math. Soc.} {\bf 50} (2007), 637-652.
   \bibitem{Co91} J. B. Conway, {\it The Theory of
Subnormal Operators}, Math. Surveys Monographs, {\bf
36}, Amer. Math. Soc. Providence, RI 1991.
  \bibitem{Dick} D. R. Dickinson, {\em Operators;
an algebraic synthesis}, Macmillan, London, 1967.
    \bibitem{GW} R. Gellar,  L. J. Wallen,
Subnormal weighted shifts and the Halmos-Bram
criterion, {\em Proc. Japan Acad.} {\bf 46} (1970),
375-378.
    \bibitem{Hal} P. R. Halmos, Ten problems in Hilbert
space, {\em Bull. Amer. Math. Soc.} {\bf 76} (1970),
887-933.
   \bibitem{Jab03} Z. Jab{\l}o\'{n}ski, Hyperexpansive composition
operators, {\em Math. Proc. Cambridge Philos. Soc.}
{\bf 135} (2003), 513-526.
   \bibitem{J-K} Z. Jab{\l}o\'{n}ski, J. Ko\'{s}mider,
$m$-isometric composition operators on a directed
graph withone circuit, preprint, 2019.
   \bibitem{Ja-St} Z. Jab{\l}o\'{n}ski, J. Stochel,
unbounded 2-hyperexpansive operators, {\em Proc. Edin.
Math. Soc.} {\bf 44} (2001), 613-629.
   \bibitem{R-0} S. Richter, Invariant subspaces of
the Dirichlet shift, {\em Jour. Reine Angew. Math.}
{\bf 386} (1988), 205-220.
   \bibitem{Rud76} W. Rudin, {\em Principles of mathematical
analysis}, International Series in Pure and Applied
Mathematics, McGraw-Hill Book Co., New York, 1976.
   \bibitem{Sh} S. Shimorin, Wold-type decompositions
and wandering subspaces for operators close to
isometries, {\em Jour. Reine Angew. Math.} {\bf 531}
(2001), 147-189.
   \bibitem{SF70} B. Sz.-Nagy, C.
Foia\c{s}, {\em Harmonic analysis of operators on
Hilbert space,} Translated from the French and revised
North-Holland Publishing Co., Amsterdam-London;
American Elsevier Publishing Co., Inc., New York;
Akad\'{e}miai Kiad\'{o}, Budapest 1970.
   \end{thebibliography}
   \end{document}